\documentclass[10pt]{article}

\usepackage{amsmath,amsthm}
\usepackage{amssymb,latexsym}
\usepackage{enumerate}
\usepackage[all]{xy}
\usepackage{multicol}

\renewcommand{\mod}{\operatorname{mod}}
\newcommand{\ind}{\operatorname{ind}}
\newcommand{\proj}{\operatorname{proj}}

\newcommand{\rad}{\operatorname{rad}}

\newcommand{\Hom}{\operatorname{Hom}}
\newcommand{\End}{\operatorname{End}}

\newcommand{\add}{\operatorname{add}}

\newcommand{\op}{\operatorname{op}}

\newcommand{\cC}{\mathcal{C}}

\newtheorem{thm}{Theorem}[section]

\newtheorem{lem}[thm]{Lemma}
\newtheorem{proposition}[thm]{Proposition}

\newtheorem*{rmk}{Remark}

\begin{document}

\begin{center} \textsc{Two tilts of higher spherical algebras} \end{center}

\begin{center} Adam Skowyrski \end{center} 

\begin{center}Faculty of Mathematics and Computer Science, \\
Nicolaus Copernicus University, Chopina 12/18, 87-100 Toru\'n, Poland \end{center}

\textbf{Abstract.} We introduce and study two exotic families of finite-dimensional algebras over an algebraically 
closed field. We prove that every such an algebra is derived equivalent to a higher spherical algebra studied by 
Erdmann and Skowro\'nski in \cite{ES4}, and hence that it is a tame symmetric periodic algebra of period $4$. This 
together with the results of \cite{ES2,ES4} shows that every trivial extension algebra of a tubular algebra of type 
$(2,2,2,2)$ admits a family of periodic symmetric higher deformations which are tame of non-polynomial growth and 
have the same Gabriel quiver, answering the question recently raised by Skowro\'nski.

MSC: 16D50, 16E30, 16G60, 18E30

Keywords: Periodic algebra, Symmetric algebra, Tame algebra, Higher spherical algebra, Derived equivalence

\section{Introduction}\label{sec:1}

Throughout this article, $K$ will denote a fixed algebraically closed field. By an algebra we mean an associative 
finite-dimensional $K$-algebra with identity, which we assume to be basic and indecomposable. For an algebra $A$ 
we denote by $\mod A$ the category of finite dimensional right $A$-modules, by $\ind A$ its full subcategory 
formed by all indecomposable modules, and by $D$ the standard duality $D=\Hom_K(-,K)$ on $\mod A$. An algebra $A$ 
is called \emph{self-injective} if $A_A$ is injective in $\mod A$, or equivalently, the projective modules in 
$\mod A$ are injective. A prominent class of self-injective algebras is provided by the \emph{symmetric algebras}, 
for which $A$ and $D(A)$ are isomorphic as $A$-$A$-bimodules. Any algebra $A$ is a quotient algebra of its trivial 
extension $T(A)=A\ltimes D(A)$, which is a symmetric algebra.

For an algebra $A$, the module category of $\mod A^e$ of its enveloping algebra $A^e=A^{\op}\otimes_K A$ is naturally 
identified with the category of $A$-$A$-bimodules. By $\Omega_{A^e}$ we denote the \emph{syzygy} operator on 
$\mod A^e$, which assigns to a module $M$ in $\mod A^e$ the kernel $\Omega_{A^e}(M)$ of a minimal projective cover 
of $M$ in $\mod A^e$. An algebra is said to be \emph{periodic} if $\Omega_{A^e}^n(A)\simeq A$ in $\mod A^e$, for 
some $n\geqslant 1$, and if so, the minimal such $n$ is called the \emph{period} of $A$. Periodic algebras are 
examples of self-injective algebras with periodic Hochschild cohomology. 

Finding and possibly classifying symmetric and periodic algebras an important problem because its interesting 
connections with group theory, topology, singularity theory and cluster algebras. For details, we refer to the 
survey article \cite{ES0} and the introductions of \cite{BES,ES1,ES3}. Dugas proved in \cite{D} that every 
representation-finite self-injective algebra, different from $K$, is a periodic algebra. The representation-infinite 
periodic algebras of polynomial growth were classified by Bia{\l}kowski, Erdmann and Skowro\'nski in \cite{BES}. 
We refer interested reader to the survey \cite{S2} for a description of these algebras.

We are concerned with the classification of all periodic symmetric tame algebras of non-polynomial growth. It is 
conjectured in \cite[Problem]{ES1} that every such an algebra is of period $4$. There is a large class of tame 
symmetric periodic algebras of period $4$, called the \emph{weighted surface algebras}, which are associated to 
triangulations of compact real surfaces. Those algebras and their deformations are investigated in \cite{ES1}--\cite{ES4}. It is also conjectured by Erdmann and Skowro\'nski (private communication) that the class of tame 
symmetric periodic algebras of period $4$ consists (up to isomorphism) of weighted surface algebras - as defined in 
\cite{ES1g}, their socle deformations (occuring only in characteristic $2$) and a finite number of exeptional 
families of algebras given by quivers with the number of vertices dividing $6$. Classifying all these exeptional 
algebras is currently and important problem. 

The article is related to the representation theory of tubular algebras \cite{HR,Rin} and the associated 
self-injective algebras of polynomial growth \cite{NS,S1}, which have been recently shown to be periodic \cite{BES}. 
In particular, the trivial extensions $T(B)$ of tubular algebras $B$ provide a wide class of periodic symmetric 
algebras of polynomial growth. Among those there are only four families of algebras of period $4$, namely 
$T(B_1(\mu))$, $T(B_2(\mu))$, $T(B_3(\mu))$ and $T(B_4(\mu))$, given by the tubular algebras $B_1(\mu)$, $B_2(\mu)$, 
$B_3(\mu)$, $B_4(\mu)$, $\mu\in K\setminus\{0,1\}$, of type $(2,2,2,2)$, described in \cite[Example 3.3]{S1}. 
It was lately observed in \cite[Lemma 6.2]{ES1} that the algebras $T(B_4(\mu))$ are the tetrahedral algebras 
(associated to the natural triangulation of the tetrahedron), and in \cite[Lemma 3.7]{ES1g} that the algebras 
$T(B_3(\mu))$ are the spherical algebras (given by a triangulation of sphere in $\mathbb{R}^2$ by four triangles). 
Further, it was proved in \cite{ES2,ES4} that both $T(B_3(\mu))$ and $T(B_4(\mu))$ admit higher deformations, 
called respectively, \emph{higher spherical} and \emph{higher tetrahedral algebras}, which form two exotic 
families of symmetric periodic tame algebras of non-polynomial growth of period $4$. We note that these algebras 
are not isomorphic to weighted surface algebras but their Gabriel quivers are the Gabriel quivers of $T(B_3(\mu))$ 
and $T(B_4(\mu))$, respectively. On the other hand, the Gabriel quivers of $T(B_1(\mu))$ and $T(B_2(\mu))$ are 
the quivers $Q_E$ and $Q_F$ presented below, so also $T(B_1(\mu))$ and $T(B_2(\mu))$ are not weighted surface 
algebras. We also recall that, by general theory of tubular algebras (see \cite[Section 1]{HR} and 
\cite[Section 5]{Rin}), the tubular algebras $B_i(\mu)$, for $i\in\{1,\dots,4\}$, are tilting equivalent and 
hence their trivial extensions $T(B_i(\mu))$, $i\in\{1,\dots,4\}$, are derived equivalent algebras \cite{Rick3}. 

The aim of the paper is to prove that the trivial extension algebras $T(B_1(\mu))$ and $T(B_2(\mu))$ also admit 
higher deformations with the Gabriel quivers $Q_E$ and $Q_F$, respectively, which form new exotic families of 
symmetric tame periodic algebras of period $4$. This answers the question posed by Skowro\'nski during his 
spring lectures in Toru\'n. 

We present now these two families of algebras. Let $m\geqslant 2$ be a natural number and $\lambda\in K^*= K 
\setminus\{0\}$. The algebra $E(m,\lambda)$ is given by the quiver $Q_E$ of the form 

$$
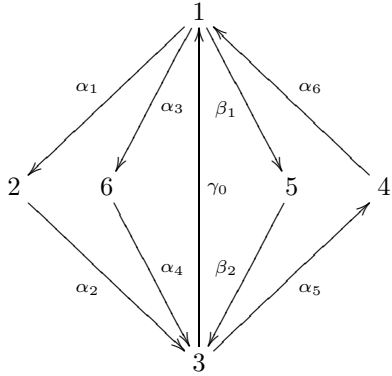
\xymatrix{&&1\ar[lldd]_{\alpha_1}\ar[ldd]^{\alpha_3}\ar[rdd]_{\beta_1}&& \\ &&&& \\
 2\ar[rrdd]_{\alpha_2} & 6\ar[rdd]^{\alpha_4} &  & 5\ar[ldd]_{\beta_2} & 4\ar[lluu]_{\alpha_6} \\ &&&& \\
 && 3\ar[uuuu]_{\gamma_0}\ar[rruu]_{\alpha_5} && }$$
 
and the relations:
\begin{enumerate}
\item[(E1)] $\alpha_1\alpha_2-\alpha_3\alpha_4+\beta_1\beta_2=0$,

\item[(E2)] $\alpha_3\alpha_4\alpha_5=\alpha_1\alpha_2\alpha_5+\lambda(\alpha_1\alpha_2\alpha_5\alpha_6)^{m-1}\alpha_1\alpha_2\alpha_5$, 

\item[(E3)] $\gamma_0\beta_1=0$, $\gamma_0\alpha_3=\alpha_5\alpha_6\alpha_3$ and $\gamma_0\alpha_1=\alpha_5\alpha_6\alpha_1+
\lambda(\alpha_5\alpha_6\alpha_1\alpha_2)^{m-1}\alpha_5\alpha_6\alpha_1$,

\item[(E4)] $\alpha_2\gamma_0=\alpha_2\alpha_5\alpha_6+\lambda(\alpha_2\alpha_5\alpha_6\alpha_1)^{m-1}\alpha_2\alpha_5\alpha_6$ and 
$$(\alpha_2\alpha_5\alpha_6\alpha_1)^{m-1}\alpha_2\alpha_5\alpha_6\alpha_3=0,$$

\item[(E5)] $\alpha_4\gamma_0=\alpha_4\alpha_5\alpha_6$ and $\alpha_4\alpha_5\alpha_6\alpha_1(\alpha_2\alpha_5\alpha_6\alpha_1)^{m-1}=0$,

\item[(E6)] $\alpha_6\alpha_3\alpha_4=\alpha_6\alpha_1\alpha_2+\lambda(\alpha_6\alpha_1\alpha_2\alpha_5)^{m-1}\alpha_6\alpha_1\alpha_2$,

\item[(E7)] $\beta_2\gamma_0=0$, $\beta_2\alpha_5\alpha_6\alpha_1=0$, $\beta_2\alpha_5\alpha_6\alpha_3=0$ and 
$\beta_2\alpha_5\alpha_6\beta_1\beta_2=0$, 

\item[(E8)] if a cycle $C$ at vertex $i$ in $Q_E$ is of length $3$ or $4$, then $\theta C^m=C^m\phi=0$, for 
all arrows $\theta,\phi$ with $t(\theta)=s(\phi)=i$.\end{enumerate}

The algebra $F(m,\lambda)$ is given by the quiver $Q_F$

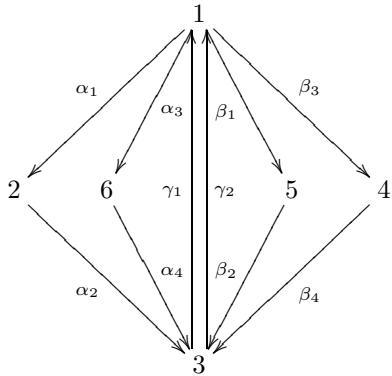
$$\xymatrix{&&1\ar[lldd]_{\alpha_1}\ar[ldd]^{\alpha_3}\ar[rdd]_{\beta_1}\ar[rrdd]^{\beta_3}&& \\ &&&& \\
 2\ar[rrdd]_{\alpha_2} & 6\ar[rdd]^{\alpha_4} &  & 5\ar[ldd]_{\beta_2} & 4\ar[lldd]^{\beta_4} \\ &&&& \\
 && 3\ar@<-0.1cm>[uuuu]_{\gamma_2}\ar@<+0.1cm>[uuuu]^{\gamma_1} && }$$ 
 
and the following relations:
\begin{enumerate}
\item[(F1)] $\alpha_1\alpha_2-\alpha_3\alpha_4+\beta_1\beta_2=0$, $\alpha_1\alpha_2-\alpha_3\alpha_4+\beta_3\beta_4+
\lambda(\alpha_3\alpha_4\gamma_1)^{m-1}\alpha_3\alpha_4=0$,

\item[(F2)] $\alpha_2\gamma_1=\alpha_2\gamma_2$ and $\gamma_2\alpha_1=\gamma_1\alpha_1$,

\item[(F3)] $\alpha_4\gamma_2=\alpha_4\gamma_1+\lambda(\alpha_4\gamma_1\alpha_3)^{m-1}\alpha_4\gamma_1$ and $\gamma_2\alpha_3=\gamma_1\alpha_3+
\lambda(\gamma_1\alpha_3\alpha_4)^{m-1}\gamma_1\alpha_3$,
 
\item[(F4)] $\beta_2\gamma_1=0$, $\beta_4\gamma_2=0$, $\gamma_1\beta_1=0$, and $\gamma_2\beta_3=0$,

\item[(F5)] $\beta_2\gamma_2\beta_1\beta_2=0$ and $\beta_4\gamma_1\beta_3\beta_4=0$,

\item[(F6)] $\gamma_1\alpha_1\alpha_2=\gamma_1\alpha_3\alpha_4$ and $\alpha_1\alpha_2\gamma_1=\alpha_3\alpha_4\gamma_1$,

\item[(F7)] $(\alpha_3\alpha_4\gamma_1)^m\alpha_1=0$, $(\alpha_3\alpha_4\gamma_1)^m\alpha_3=0$, 
$\gamma_1(\alpha_3\alpha_4\gamma_1)^m=0$, $\alpha_2(\gamma_1\alpha_3\alpha_4)^m=0$ and 
$\alpha_4(\gamma_1\alpha_3\alpha_4)^m=0$. \end{enumerate}

Following \cite{ES4}, we denote by $S(m,\lambda)$ the higher spherical algebra (see Section \ref{sec:2} for details). 
The following theorem is the main result of the article.

\begin{thm}\label{mainthm} For any $m\geqslant 2$ and $\lambda\in K^*$, the algebras $E(m,\lambda)$, $F(m,\lambda)$ 
and $S(m,\lambda)$ are derived equivalent.\end{thm} 

Applying general facts on derived equivalences of algebras (see theorems presented in Section \ref{sec:3}), we obtain 
the following theorem.

\begin{thm}\label{cor:1} Let $E=E(m,\lambda)$ and $F=F(m,\lambda)$, for some $m\geqslant 2$ and $\lambda\in K^*$. 
Then the following statements hold:
\begin{enumerate}
\item[(1)] $E$ and $F$ are finite dimensional algebras with 
$$\dim_K E(m,\lambda)=25m+13\mbox{ and }\dim_K F(m,\lambda)=16m+16.$$
\item[(2)] $E$ and $F$ are symetric algebras.
\item[(3)] $E$ and $F$ are periodic algebras of period $4$.
\item[(4)] $E$ and $F$ are tame algebras of non-polynomial growth,
\end{enumerate}\end{thm} 

We also mention that families $E(m,\lambda)$ and $F(m,\lambda)$ reveal new exotic pieces of classification problem 
for algebras of generalized quaterion type \cite{ES3} as well as of solution for the problem raised in \cite{ES4}. 
For the relevant background on the representation theory of associative algebras we refer to \cite{ASS, Ha, SS, SY}. 

\section{Higher spherical algebras}\label{sec:2} This section is devoted to outline basic facts about higher 
spherical algebras \cite{ES4} needed in further considerations.

For a natural number $m\geqslant 2$ and $\lambda\in K^*$ we denote by $S(m,\lambda)$ the algebra given by the 
quiver of the following form 
$$
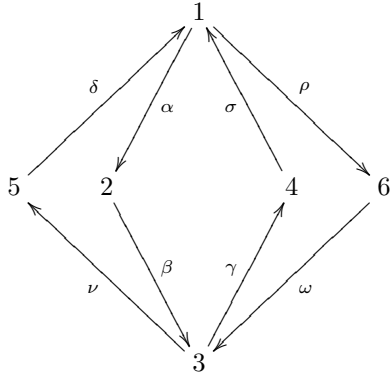
\xymatrix{&&1\ar[ldd]^{\alpha}\ar[rrdd]^{\rho}&& \\ &&&& \\
 5\ar[rruu]^{\delta} & 2\ar[rdd]^{\beta} &  & 4\ar[luu]^{\sigma} & 6\ar[lldd]^{\omega} \\ &&&& \\
 && 3\ar[lluu]^{\nu}\ar[ruu]^{\gamma} && }$$ 
and the relations: 
\begin{multicols}{2}
\begin{enumerate}
\item[(1)] $\beta\nu\delta=\beta\gamma\sigma+\lambda(\beta\gamma\sigma\alpha)^{m-1}\beta\gamma\sigma$,
\item[(2)] $\nu\delta\alpha=\gamma\sigma\alpha+\lambda(\gamma\sigma\alpha\beta)^{m-1}\gamma\sigma\alpha$,
\item[(3)] $\sigma\rho\omega=\sigma\alpha\beta+\lambda(\sigma\alpha\beta\gamma)^{m-1}\sigma\alpha\beta$,
\item[(4)] $\rho\omega\gamma=\alpha\beta\gamma+\lambda(\alpha\beta\gamma\sigma)^{m-1}\alpha\beta\gamma$,
\item[(5)] $\alpha\beta\nu=\rho\omega\nu$,
\item[(6)] $\delta\alpha\beta=\delta\rho\omega$,
\item[(7)] $\omega\gamma\sigma=\omega\nu\delta$,
\item[(8)] $\gamma\sigma\rho=\nu\delta\rho$,
\item[(9)] $(\alpha\beta\gamma\sigma)^m\alpha=0$,
\item[(10)] $(\gamma\sigma\alpha\beta)^m\gamma=0$.
\end{enumerate}
\end{multicols}

For further calculations we will use some additional relations in $S(m,\lambda)$. It follows from 
\cite[Lemma 4.1]{ES4} that the following relations hold in $S(m,\lambda)$: 

\begin{enumerate}
\item[(1')] $(\beta\gamma\sigma\alpha)^{m-1}\beta\gamma\sigma\rho=0$ and $(\alpha\beta\gamma\sigma)^m\rho=0$,
\item[(2')] $\omega\gamma\sigma\alpha(\beta\gamma\sigma\alpha)^{m-1}=0$ and $\omega(\gamma\sigma\alpha\beta)^m=0$,
\item[(3')] $(\sigma\alpha\beta\gamma)^{m-1}\sigma\alpha\beta\nu=0$ and $(\gamma\sigma\alpha\beta)^m\nu=0$,
\item[(4')] $\delta\alpha\beta\gamma(\sigma\alpha\beta\gamma)^{m-1}=0$ and $\delta(\alpha\beta\gamma\sigma)^m=0$,
\item[(5')] $(\rho\omega\nu\delta)^r=(\alpha\beta\gamma\sigma)^r$ and $(\nu\delta\rho\omega)^r=(\gamma\sigma\alpha\beta)^r$, for any $r\in\{2,\dots,m\}$. \end{enumerate}

In the rest part of the paper we will frequently refer to above relations and relations for $E(m,\lambda)$ and 
$F(m,\lambda)$ presented in the introduction.

Let $S=S(m,\lambda)$ be a fixed higher spherical algebra, with $m\geqslant 2$ and $\lambda\in K^*$. Denote by 
$e_1,\dots,e_6$ the complete set of all pairwise orthogonal primitive idempotents in $S$. The modules $P_1=e_1S,
\dots,P_6=e_6S$ form a complete set of pairwise nonisomorphic indecomposable projective $S$-modules. For $i\in 
\{1,2,3,4\}$, let $X_i$ be the unique cycle of length $4$ in $Q_S$ around $i$ formed by the arrows $\alpha,
\beta,\gamma,\sigma$. Similarily, if $i\in\{5,6\}$, we denote by $X_i$ the unique cycle of length $4$ in $Q_S$ 
around $i$ formed by the arrows $\rho,\omega,\nu,\delta$.

\begin{rmk} For any cycle $X'_i$ of length $4$ in $Q_S$ around $i$, the difference $X_i-X'_i$ in $S$ is zero or it 
belongs to $K X_i^m$ and $X_i^k=(X'_i)^k$, for $k\in\{2,\dots,m\}$. Moreover, from relations (1)-(10) and (1')-(5') above, we obtain that $\theta X_i^m=X_i^m\psi=0$, for all arrows $\theta,\psi$ in $Q_S$ with $t(\theta)=s(\psi)=i$. 
Note also that, if $i\in\{1,\dots,6\}$, then any two paths $p_1,p_2$ from $i$ to $j$ in $Q_S$ give $X_i^kp_1= 
X_i^kp_2$ and $p_1X_j^k=p_2X_j^k$ in $S$, if $k\in \{1,\dots,m\}$. Further observations are based directly on 
\cite[see proof of Proposition 4.2]{ES4}. Consider the following paths in $Q_S$ 

$$ X_{12}=\alpha, X_{13}=\alpha\beta, X_{14}=\alpha\beta\gamma, X_{15}=\alpha\beta\nu, X_{16}=\rho, \tilde{X}_{13}=
\rho\omega$$
$$ X_{21}=\beta\gamma\sigma, X_{23}=\beta, X_{24}=\beta\gamma, X_{25}=\beta\nu, X_{26}=\beta\gamma\sigma\rho,$$
$$ X_{31}=\gamma\sigma, X_{32}=\gamma\sigma\alpha, X_{34}=\gamma, X_{35}=\nu, X_{36}=\gamma\sigma\rho, 
\tilde{X}_{31}=\nu\delta$$ 
$$ X_{41}=\sigma, X_{42}=\sigma\alpha, X_{43}=\sigma\alpha\beta, X_{45}=\sigma\alpha\beta\nu, X_{46}=\sigma\rho,$$
$$ X_{51}=\delta, X_{52}=\delta\alpha, X_{53}=\delta\rho\omega, X_{54}=\delta\rho\omega\gamma, X_{56}=\delta\rho, $$ 
$$ X_{61}=\omega\nu\delta, X_{62}=\omega\nu\delta\alpha, X_{63}=\omega, X_{64}=\omega\gamma, X_{65}=\omega\nu$$
and write $X_{ij}^r$ for path of the form $X_i^rX_{ij}$. For any $i,j\in\{1,\dots,6\}$, we pick a $K$-basis 
$\mathcal{S}_{ij}$ of $e_iSe_j$ as described below. For $i\in\{1,\dots,6\}$, we take $\mathcal{S}_{ii}=\{X_i^r;\mbox{ }0\leqslant r \leqslant m\}$, whereas
$$\mathcal{S}_{13}=\{X_{13}^r;\mbox{ }0\leqslant r\leqslant m-1\}\cup\{\tilde{X}_{13}\}\mbox{ and }\mathcal{S}_{31}=\{X_{31}^r;\mbox{ }0\leqslant r\leqslant m-1\}\cup\{\tilde{X}_{31}\}.$$ 
Moreover, if $(i,j)$ is one of pairs $(2,6)$, $(6,2)$, $(4,5)$ or $(5,4)$, then we put $\mathcal{S}_{ij}=\{X_{ij}^r;
\mbox{ }0\leqslant r\leqslant m-2\}$. For the remaining pairs of indices $\mathcal{S}_{ij}=\{X_{ij}^r;\mbox{ }
0\leqslant r\leqslant m-1\}$.\end{rmk}

In particular, it follows that the Cartan matrix of $S=S(m,\lambda)$ is of the form 
$$C_S=[c^S_{ij}]=\left[\begin{array}{cccccc}
m+1 & m & m+1 & m & m & m \\ m & m+1 & m & m & m & m-1 \\ m+1 & m & m+1 & m & m & m \\ 
m & m & m & m+1 & m-1 & m \\ m & m & m & m-1 & m+1 & m \\ m & m-1 & m & m & m & m+1 \end{array}\right].$$ 

Finally, we recall one of the main results from \cite[see Theorem 1]{ES4}.

\begin{thm}\label{thm:2.1} Let $S=S(m,\lambda)$ be a higher spherical algebra. Then the following statements hold
\begin{enumerate}[(1)]
\item $S$ is a finite dimensional algebra with $\dim_KS=36m+4$.
\item $S$ is a tame symmetric algebra of non-polynomial growth.
\item $S$ is a periodic algebra of period $4$.
\end{enumerate}\end{thm}

\section{Derived equivalences of algebras}\label{sec:3} In this section, we recall some basic facts on derived 
equivalences of algebras.

For an algebra $A$ we denote by $\mathcal{C}_A$ the category of bounded complexes $X=(X^n,d^n)_{n\in\mathbb{Z}}$ ($X^n$ are modules in $\mod A$ and $d^n=d^n_X:X^n\to X^{n-1}$ homomorphisms of $A$-modules) with usually understood 
morphisms between them. Two morphisms $f,g:X\to Y$ in $\mathcal{C}_A$ are said to be \emph{homotopic}, which is 
denoted by $f\sim g$, provided that there is a collection $(s_n)_{n\in\mathbb{Z}}$ of homomorphisms 
$s_n:X^n\to Y^{n+1}$ in $\mod A$ with $f^n-g^n=s_{n-1}d^n_X+d^{n+1}_Ys_n$, for any $n\in\mathbb{Z}$. By $\mathcal{K}^b(\mod A)$ we denote the homotopy category of bounded complexes (over $A$), whose objects are the same as in 
$\mathcal{C}_A$, whereas morphisms are homotopy classes of morphisms in $\mathcal{C}_A$. Further, we recall that the 
derived category $\mathcal{D}^b(\mod A)$ of $A$ is the localization of $\mathcal{K}^b(\mod A)$ with respect to 
quasi-isomorphisms. It is well-known that both $\mathcal{K}^b(\mod A)$ and $\mathcal{D}^b(\mod A)$ admit structure 
of a triangulated $K$-category with suspension functor being the left shift functor $(-)[1]$. Two algebras $A$ and 
$B$ are called \emph{derived equivalent} if and only if the derived categories $\mathcal{D}^b(\mod A)$ and 
$\mathcal{D}^b(\mod B)$ are equivalent as triangulated categories. 

Note also that the full subcategory $\mathcal{K}_A:=\mathcal{K}^b(\proj A)$ of $\mathcal{K}^b(\mod A)$ formed by all 
(bounded) complexes of projective modules in $\mod A$ has a structure of triangulated category induced from the 
shift in homotopy category. Moreover, if $T$ is a complex in $\mathcal{K}_A$, then $T$ is said to be 
\emph{a tilting complex}, if the following conditions are satisfied:

\begin{enumerate}
\item[(T1)] $\Hom_{\mathcal{K}_A}(T,T[i])=0$, for all integers $i\neq 0$,
\item[(T2)] $\add(T)$ generates $\mathcal{K}_A$ as triangulated category.
\end{enumerate}

Further we have the following handy criterion for verifying derived equivalence of algebras, due to J. Rickard \cite[Theorem 6.4]{Rick1}. 

\begin{thm}\label{thm:3.1} Two algebras $A$ and $B$ are derived euivalent if and only if there exists a tilting 
complex $T$ in $\mathcal{K}_A$ such that $\End_{\mathcal{K}_A}(T)\cong B$. \end{thm}

Finally, thanks to Happel \cite[see III.1.3-1.4]{Ha} we have the following alternating sum formula allowing to 
compute dimensions of morphisms spaces in homotopy category between summands of tilting complexes. 

\begin{thm}\label{thm:3.2} Let $A$ be an algebra and $Q,R$ two complexes in $\mathcal{K}_A$ for which $\Hom_{\mathcal{K}_A}(Q,R[i])=0$, for all 
integers $i\neq 0$. Then
$$\dim_K\Hom_{\mathcal{K}_A}(Q,R)=\sum_{r,s\in\mathbb{Z}}(-1)^{r-s}\dim_K\Hom_A(Q^r,R^s).$$ \end{thm}

Note that the following two theorems (see \cite[Corollary 5.3]{Rick3} and \cite[Theorem 2.9]{ES0}) are of 
importance.

\begin{thm}\label{thm:3.3} Let $A$ and $B$ be two derived equivalent algebras. Then $A$ is symmetric if and only if $B$ is symetric.\end{thm}

\begin{thm}\label{thm:3.4} Let $A$ and $B$ be two derived equivalent algebras. Then $A$ is periodic if and only if $B$ is periodic, and if this is the case, then they have the same period.\end{thm} 

Because in the class of self-injective algebras, derived equivalence implies stable equivalence (see \cite[Corollary 2.2]{Rick2} or \cite[Corollary 5.3]{Rick3}), one may conclude from \cite[Theorems 4.4 and 5.6]{CB} and 
\cite[Corollary 2]{KZ} that the following theorem holds. 

\begin{thm}\label{thm:3.5} Let $A$ and $B$ be two derived equivalent self-injective algebras. Then the following 
equivalences hold.
\begin{enumerate}
\item[(1)] $A$ is tame if and only if $B$ is tame.
\item[(2)] $A$ is of polynomial growth if and only if $B$ is of polynomial growth.
\end{enumerate}\end{thm}

\section{Two tilting complexes over higher spherical algebras} 

In this section we define two particular complexes over a given higher spherical algebra $S=S(m,\lambda)$, with 
$m\geqslant 2,\lambda\in K^*$, and we prove that these are tilting complexes in $\mathcal{K}_S$. 

We stick to the notations settled in Section \ref{sec:2} for idempotents $e_i\in S$ and projective modules 
$P_i=e_iS$ in $\ind S$ (also for elements of their bases). We identify arrows $\rho:i\to j$ in $Q_S$ (or in general, 
elements in $e_iSe_j$) with associated homomorphisms $\rho:P_j\to P_i$ of indecomposable projective $S$-modules, 
given by left multiplication by $\rho$. For every $i\in\{1,\dots,6\}$, we will denote by $\textbf{P}_i$ the stalk 
complex in $\mathcal{K}_S$ of the form 
$$0\to P_i\to 0$$
concentrated in degree 0. Moreover, consider the following two 2-term complexes $\bar{\textbf{P}}_4$ and 
$\bar{\textbf{P}}_5$ concentrated in degrees 1 and 0
$$\bar{\textbf{P}}_4: \xymatrix{0 \ar[r] & P_4\ar[r]^{\gamma} & P_3\ar[r] & 0}\quad\mbox{and}\quad \bar{\textbf{P}}_5: \xymatrix{0 \ar[r] & P_5\ar[r]^{\nu} & P_3\ar[r] & 0}.$$

The following technical lemma describes morphisms in $\mathcal{K}_S$ between defined complexes and their shifts. 
In most cases a morphism in $\mathcal{K}_S$ between those complexes is given by a single homomorphism 
$P_i\to P_j$, $i,j\in\{1,\dots,6\}$, in $\proj S$ identified with an element (a combination of paths) in $e_jSe_i$. 

\begin{lem}\label{tech:3.2}
\begin{enumerate}
\item[(a)] For any $\psi$ in $e_4Se_i$ (respectively, in $e_5Se_i$), the corresponding homomorphism $f_0=\psi:P_i\to P_4$ (respectively, $P_i\to P_5$) induces morphism $\textbf{P}_i\to\bar{\textbf{P}}_4[-1]$ (respectively, morphism 
$\textbf{P}_i\to\bar{\textbf{P}}_5[-1]$) in $\mathcal{C}_S$ if and only if $\gamma\psi=0$ (respectively, $\nu\psi=0$) 
in $S$; moreover, this holds iff $\psi\in KX_4^m$ (respectively, $\psi\in KX_5^m$). 

\item[(b)] Every element $\psi\in e_3Se_i$ induces morphism $\textbf{P}_i\to\bar{\textbf{P}}_4$ (respectively, morphism $\textbf{P}_i\to \bar{\textbf{P}}_5$) in $\mathcal{C}_S$ and such a morphism is homotopic to $0$ if and only if $\psi\in \gamma S$ (respectively, $\psi\in \nu S$).

\item[(c)] Let $X$ be a complex in $\mathcal{C}_S$ with $X^1=P_i$, $i\in\{1,\dots,6\}$, and $X^n=0$, for $n\leqslant 0$. Then every element $\psi\in e_iSe_4$ (respectively, in $e_iSe_5$) induces morphism $f=f_1:\bar{\textbf{P}}_4\to 
X$ (respectively, morphism $\bar{\textbf{P}}_5\to X$) in $\mathcal{C}_S$ and $f\simeq 0$ if $\psi\in S\gamma$ 
(repectively, $\psi\in S\nu$). 

\item[(d)] For $\rho_0\in e_3Se_3$ and $\rho_1\in e_5Se_4$ (respectively, $\rho_1\in e_4Se_5$), the pair $f=(\rho_1,\rho_0)$ defines a morphism $\bar{\textbf{P}}_4\to \bar{\textbf{P}}_5$ (respectively, $\bar{\textbf{P}}_5\to 
\bar{\textbf{P}}_4$) in $\mathcal{C}_S$ iff $\nu\rho_1=\rho_0\gamma$ (respectively, $\gamma\rho_1=\rho_0\nu$) and 
in the case, $f\simeq 0$ iff there is $\psi '\in e_5Se_3$ with $\rho_0=\nu\psi '$ and $\rho_1=\psi '\gamma$ 
(respectively, there is $\psi '\in e_4Se_3$ with $\rho_0=\gamma\rho '$ and $\rho_1=\rho '\nu$). 

\item[(e)] An element $\phi\in e_iSe_3$ induces morphism in $\mathcal{C}_S$ of the form $\bar{\textbf{P}}_4
\to\textbf{P}_i$ (respectively, of the form $\bar{\textbf{P}}_5\to\textbf{P}_i$) if and only if $\phi\gamma =0$ 
(respetively, $\phi\nu=0$) in $S$. This is the case iff $i=1$ and $\phi\in K(\rho\omega -\alpha\beta-\lambda(\alpha
\beta\gamma\sigma)^{m-1}\alpha\beta)$ (respectively, $\phi\in K(\rho\omega -\alpha\beta)$) or $\rho\in KX_3^m$.
\end{enumerate}\end{lem} 

\begin{proof} Observe that conditions b)-d) and first equivalences in a) and e) follow immediately from definition of morphisms in $\mathcal{C}_S$ and homotopy relation (note also that composition of morphisms corresponds to 
appropriate multiplying of elements in $S$). Let us sketch the proof of remaining equivalences in a) and e). Observe 
first that, if $i\in\{2,6\}$, then for any $\phi\in e_iSe_3$ both $\phi\gamma$ and $\phi\nu$ are nonzero. Next, for 
$i=3$, it is also easy to check that the only generator $\phi=X_3^k$ of $e_3Se_3$ with $\phi\gamma=\phi\nu=0$ is 
$\phi=X_3^m$, so it remains to consider the case $i=1$. Then, for any base element $\phi=X_{13}^k$ with $k\in 
\{0,\dots,m-1\}$ we have $\phi\gamma=X_{14}^k$ and $\phi\nu=X^k_{15}$, whereas $\tilde{X}_{13}\gamma=X^0_{14}+
\lambda X^{m-1}_{14}$ and $\tilde{X}_{13}\nu=X^0_{15}$. As a result, for any element $\phi\in e_1 Se_3$, we have 
$\phi\nu=0$ if and and only if $\phi\in K(\rho\omega-\alpha\beta)$. In a very similar way, we deduce that 
$\phi\gamma=0$ if and only if $\phi$ belongs to $K(\tilde{X}_{13}-X^0_{13}-\lambda X^{m-1}_{13})$, which finishes 
the proof of e). It is straightforward calculation to check that for any $\phi$ in $e_4Se_i$ we have: $\gamma\phi=0$ 
if and only if $i=4$ and $\phi\in KX_4^m$. Similarily, for $\phi\in e_5Se_i$.\end{proof} 

Now, we may almost directly deduce the following proposition. 

\begin{proposition}\label{prop:4.2} The following two complexes 
$$T^{(1)}=\textbf{P}_1\oplus \textbf{P}_2 \oplus \textbf{P}_3 \oplus \textbf{P}_4 \oplus \bar{\textbf{P}_5} \oplus \textbf{P}_6\quad\mbox{and}\quad 
T^{(2)}=\textbf{P}_1\oplus \textbf{P}_2 \oplus \textbf{P}_3 \oplus \bar{\textbf{P}_4} \oplus \bar{\textbf{P}_5} \oplus \textbf{P}_6$$ 
in $\mathcal{K}_S$ are tilting complexes. \end{proposition}

\begin{proof} First we prove that $T^{(2)}$ is a tilting complex. Denote by $T_1,\dots,T_6$ its indecomposable direct 
summands in order fixed above. Then it follows directly from a) above that 
$\Hom_{\mathcal{C}_S}(T_k,(T_4\oplus T_5)[-1])=0$, for any $k\neq 4$ and $5$, hence also 
$\Hom_{\cC_S}(T_1\oplus T_2\oplus T_3\oplus T_6,(T_4\oplus T_5)[i])=0$, for $i < 0$, and consequently, for all 
integers $i\neq 0$, because $(T_4\oplus T_5)[i]$, for $i>0$, is concentrated in positive degrees, while $T_k$ with 
$k\neq 4,5$, in degree 0. Trivially, we have $\Hom_{\cC_S}(T_k,T_j[i])=0$, for any integer $i\neq 0$ and $k,j\in 
\{1,2,3,6\}$, so $\Hom_{\cC_S}(T_1\oplus T_2\oplus T_3\oplus T_6,T^{(2)}[i])=0$, for all $i\neq 0$. 

Further, if any complex $X$ of the form $X=T_k[1]$ satisfies $X^n=0$, for all $n\leqslant 0$ and $X^1=P_t$ with 
$t\neq 4,5$. But $e_3Se_4=e_3(\rad S)e_4\subset S\gamma$, as well as $e_3Se_5=e_3(\rad S)e_5 \subset S\nu$, because 
all nontrivial paths ending at $4$ (or $5$) pass through $\gamma$ (respectively, through $\nu$), so we conclude from 
c) in the previous Lemma that $\Hom_{\mathcal{K}_S}(T_4\oplus T_5,T^{(2)}[1]))=0$, and hence we have 
$\Hom_{\mathcal{K}_S}(T_4\oplus T_5,T^{(2)}[i])=0$, for any $i>0$. Obviously, if $k\neq 4,5$, we obtain 
$\Hom_{\cC_S}(T_4\oplus T_5,T_k[i])=0$, for any $i<0$, and finally $\Hom_{\cC_S}(T_k,T_l[-1])=0$, for all 
$k,l\in\{4,5\}$, since every such chain map $T_k\to T_l[-1]$ is given by an element $\psi\in e_lSe_3= e_l(\rad S)e_3$ 
with $\gamma\psi\neq 0$ or $\nu\psi\neq 0$, if $l=4$ or $5$, respectively; see also Lemma \ref{tech:3.2}a). It 
follows that $\Hom_{\mathcal{K}_S}(T_4\oplus T_5,T^{(2)}[i])=0$, also for all $i<0$. Concluding, we obtain that 
$\Hom_{\mathcal{K}_S}(T^{(2)},T^{(2)}[i])=0$, for any integer $i\neq 0$, and consequently, $T^{(2)}$ satisfies 
condition (T1) from definition of a tilting complex. 

Now, it remains to see that the identity homomorphism on $P_3$ induce two morphisms (in degree 0) of complexes 
$f:T_3\to T_4$ and $g:T_3\to T_5$ and mapping cones $C_f$ and $C_g$ (of $f$ and $g$), are complexes in 
$\mathcal{K}_S$, of the following forms 
$$\xymatrix{P_3\oplus P_4\ar[r]^{[1\mbox{ }\gamma]} & P_3}\quad\mbox{and}\quad\xymatrix{P_3\oplus 
P_5\ar[r]^{[1\mbox{ }\nu]} & P_3}$$
But then projections $P_3\oplus P_4\to P_4$ and $P_3\oplus P_5\to P_5$ induce morphisms (in degree 1) of the forms 
$f':C_f\to\textbf{P}_4[1]$ and $g':C_g\to\textbf{P}_5[1]$, for which cones $C_{f'}$ and $C_{g'}$ are easy to check 
being exact, so $f'$ and $g'$ are quasi-isomorphisms, and thus homotopy equivalences (we deal with bounded above and 
below complexes of projective $S$-modules). This means that $C_f\simeq \textbf{P}_4[1]$ and $C_g\simeq 
\textbf{P}_5[1]$ in $\mathcal{K}_S$, and therefore triangulated category generated by $\add(T^{(2)})$ contains all 
stalk complexes $\textbf{P}_i$, for $i\in\{1,\dots,6\}$, together with their shifts. Consequently, $\add(T^{(2)})$ 
generates $\mathcal{K}_S$ as a triangulated category, and hence $T^{(2)}$ is indeed a tilting complex in 
$\mathcal{K}_S$. 

Let now $T_1,\dots,T_6$ be all indecomposable direct summands of $T^{(1)}$ ordered as in the statement. Using 
analogues of above arguments one can easily prove that also $T^{(1)}$ satisfies condition (T2). For (T1), observe 
that $\Hom_{\mathcal{K}_S}(T_k,T_l[i])=0$, for all $k,l\neq 5$ and $i\neq 0$, because all these complexes are stalk 
in degree 0. Finally, note that every chain map $\tilde{\phi}:T_5\to T_k[1]$ is given by an element $\phi$ in $e_k 
Se_5$, hence for $k\neq 5$ we have $\phi\in S\nu$, and so $\tilde{\phi}\sim 0$, by Lemma \ref{tech:3.2}c). This 
shows $\Hom_{\mathcal{K}_S}(T_5,T_k[i])=0$, for all $k\neq 5$ and $i\neq 0$. Including the above considerations for 
$T^{(2)}$, we have $\Hom_{\mathcal{K}_S}(T_5,T_5[i])=0$, for $i\neq 0$, thus $\Hom_{\mathcal{K}_S}(T_5,
T^{(1)}[i])=0$, for any integer $i\neq 0$. Direct calculation on bases of spaces $e_5 Se_k$, with $k\neq 5$, shows 
that for any nonzero $\phi\in e_5 Se_k$ we have $\nu\phi \neq 0$. As a result, we obtain 
$\Hom_{\mathcal{K}_S}(T_k,T_5[-1])=0$, and consequently, $\Hom_{\mathcal{K}_S}(T_k,T_5[i])=0$, for any $k\neq 5$ and 
$i\neq 0$. It is now clear that $\Hom_{\mathcal{K}_S}(T^{(1)},T^{(1)}[i])=0$, for all integers $i\neq 0$, that is, 
$T^{(1)}$ satisfies (T1), and therefore it is indeed a tilting complex in $\mathcal{K}_S$.\end{proof}

\section{Proof of Theorem \ref{mainthm}} Since we have already constructed appropriate tilting complexes in 
$\mathcal{K}_S$, we conclude from Theorem \ref{thm:3.1} that the algebra $S=S(m,\lambda)$ is derived equivalent to 
both algebras $E:=\End_{\mathcal{K}_S}(T^{(1)})$ and $F:=\End_{\mathcal{K}_S}(T^{(2)})$. According to Theorems 
\ref{thm:3.3} and \ref{thm:3.4}, it remains to show that there are isomorphisms $E\cong E(m,\lambda)$ and 
$F\cong F(m,\lambda)$ of $K$-algebras, and then Theorems \ref{mainthm} and \ref{cor:1} will automatically follow 
(showing isomorphisms we will also compute dimensions of $E$ and $F$). 

The following two lemmas are devoted to compute Cartan matrices of algebras $E$ and $F$. 

\begin{lem} The Cartan matrix $C_{F}$ of $F$ is of the form
$$\left[ \begin{array}{cccccc} 
m+1&m&m+1&1&1&m \\
m&m+1&m&0&0&m-1 \\
m+1&m&m+1&1&1&m \\
1&0&1&2&0&0 \\
1&0&1&0&2&0 \\
m&m-1&m&0&0&m+1 \end{array}\right].$$
\end{lem}

\begin{proof} Denote by $c_{ij}$ the ($i$,$j$) entry of $C_F$ and by $T_1,\dots,T_6$ all indecomposable direct 
summands of tilting complex $T=T^{(2)}$, ordered as presented in Proposition \ref{prop:4.2}. By definition $E$ is 
the endomorphism algebra of $T$, so modules $\tilde{P}_i:=\Hom_{\mathcal{K}_S}(T,T_i)$, for $i\in\{1,\dots,6\}$, form 
a complete set of pairwise non-isomorphic indecomposable projective modules in $\ind F$, and moreover there are 
isomorphisms of vector spaces $\Hom_{F}(\tilde{P}_i,\tilde{P}_j)\simeq\Hom_{\mathcal{K}_S}(T_i,T_j)$. Consequently, 
for all $i,j\in \{1,\dots,6\}$ we have $c_{ij}=\dim_K\Hom_{\mathcal{K}_S}(T_i,T_j)$ and complexes $T_i,T_j$ satisfy 
assumptions of Theorem \ref{thm:3.2}. Therefore we can directly calculate ($c^S_{ij}$ denote the corresponding entry 
of Cartan matrix $C_S$ of $S$) that: 
\begin{itemize}
\item $c_{ij}=\dim_K\Hom_{\mathcal{K}_S}(\textbf{P}_i,\textbf{P}_j)=\dim_K\Hom_S(P_i,P_j)=c^S_{ij}$, for any $i,j\in\{1,2,3,6\}$,
\item for $i\in\{4,5\}$ and $j\in\{1,2,3,6\}$, we have  
$$c_{ij}=\dim_K\Hom_{\mathcal{K}_S}(\bar{\textbf{P}}_i,\textbf{P}_j)=\dim_K\Hom_S(P_3,P_j)-\dim_K\Hom_S(P_i,P_j)=$$ 
$$=c^S_{3j}-c^S_{ij},\qquad\qquad\qquad\qquad\qquad\qquad\qquad\qquad\qquad\qquad\qquad\qquad$$
\item similarily, $c_{ij}=\dim_K\Hom_{\mathcal{K}_S}(\textbf{P}_i,\bar{\textbf{P}}_j)=c^S_{i3}-c^S_{ij}$, if $i\in\{1,2,3,6\}$ and $j\in\{4,5\}$ 
\item and $c_{ij}=c^S_{33}-c^S_{3j}-c^S_{i3}+c^S_{ij}$, for $i,j\in\{4,5\}$.
\end{itemize}
As a result, for any $j\in\{1,2,3,6\}$, we obtain
$$c_{4j}=c^S_{3j}-c^S_{4j}=c^S_{3j}-c^S_{5j}=c_{5j}=\left\{\begin{array}{l}1,\quad\mbox{ when }j\in\{1,3\} \\ 0,\mbox{when }j\in\{2,4\}\end{array}\right.$$ 
and, if $i\in\{1,2,3,6\}$, then $c_{i4}=c^S_{i3}-c^S_{i4}=c^S_{i3}-c^S_{i5}=c_{i5}=c_{4i}$.
 
Finally, simple chcecking shows that 
$$ c_{44}=c^S_{33}-c^S_{34}-c^S_{43}+c^S_{44}=c_{55}=$$
$$=m+1-m-m+m+1=2\qquad$$
and $c_{45}=c_{54}=m+1-m-m+m-1=0$, which completes the proof.\end{proof}

\begin{lem} The Cartan matrix $C_{E}$ of $E$ is of the form
$$\left[ \begin{array}{cccccc} 
m+1&  m &m+1 &  m & 1 & m \\
m  & m+1& m  &  m & 0 &m-1\\
m+1&  m &m+1 &  m & 1 & m \\
m  &  m & m  & m+1& 1 & m \\
1  &  0 & 1  &  1 & 2 & 0 \\
m  & m-1& m  &  m & 0 & m+1 \end{array}\right]$$
\end{lem}

\begin{proof} As before, we denote by $c_{ij}$ the corresponding entry of $C_E$. Obviously, it follows from Theorem 
\ref{thm:3.2}, that for stalk complexes $T_i,T_j$, with $i,j\neq 5$, we have $\dim_K \Hom_{\mathcal{K}_S}(T_i,T_j)= 
\dim_K\Hom_S(P_i,P_j)$, and therefore $c_{ij}=c^S_{ij}$. For any $j\neq 5$, we deduce also from Theorem \ref{thm:3.2} 
that 
$$c_{5j}=\dim_K\Hom_{\mathcal{K}_S}(\bar{\textbf{P}}_5,\textbf{P}_j)=\dim_K\Hom_S(P_3,P_j)-\dim_K\Hom_S(P_5,P_j)=$$
$$=c^S_{3j}-c^S_{5j},\qquad\qquad\qquad\qquad\qquad\qquad\qquad\qquad\qquad\qquad\qquad\qquad$$ 
and so the fifth row of $C_E$ is of the form $\left[\begin{array}{cccccc} 1 & 0 & 1 & 1 & 2 & 0 \end{array}\right]$ (it follows from the last proof that $c_{55}=\dim_K\Hom_{\mathcal{K}_S}(T_5,T_5)=2$). In a similar way, applying 
Theorem \ref{thm:3.2} to complexes $T_{i\neq 5}$ and $T_5$, we deduce that $c_{i5}=c^S_{i3}-c^S_{i5}$, so the fifth 
column is equal to transpose of fifth row, and the proof is now complete.\end{proof} 

In the next two lemmas we shall investigate Gabriel quivers. 
  
\begin{lem} The Gabriel quivers of $F$ and $F(m,\lambda)$ coincide. \end{lem}

\begin{proof} We fix order (and labeling) of vertices in $Q_F$ induced by order of direct summands $T_1,\dots,T_6$ of 
tilting complex $T=T^{(2)}$. For $i,j\in \{1,\dots,6\}$ let $F_{ij}$ denote the $K$-vector space $U/V$, where $U$ is 
a subspace of $\Hom_{\mathcal{K}_S}(T_i,T_j)$ formed by all non-isomorphisms and $V$ its subspace containing all 
compositions of two nonisomorphisms $T_i\to T'\to T_j$ with $T'$ an object in $\add(T)$. It is well known that 
$F_{ij}$ is isomorphic to $\rad_F(\tilde{P}_i,\tilde{P}_j)/\rad_F^2(\tilde{P}_i, 
\tilde{P}_j)\simeq e_j(\rad F)e_i/e_j(\rad^2F)e_j$, so in particular, the number $q_{ij}$ of arrows $i\to j$ in $Q_F$ 
is equal to $\dim_KF_{ji}$. 

First, we consider arrows $\alpha,\beta,\rho,\omega$ in $Q_S$ and the corresponding homomorphisms $\alpha:P_2\to 
P_1$, $\beta:P_3\to P_2$, $\rho:P_6\to P_1$ and $\omega:P_3\to P_6$ in $\proj S$. It is clear that then the induced 
morphisms of stalk complexes $\tilde{\alpha_1}:T_2\to T_1$, $\tilde{\alpha_2}:T_3\to T_2$, $\tilde{\alpha_3}:T_6\to 
T_1$ and $\tilde{\alpha_4}:T_3\to T_6$ give unique generators of spaces $F_{21}$, $F_{32}$, $F_{61}$ and $F_{36}$, so 
$q_{12}=q_{23}=q_{16}=q_{63}=1$. For instance, any nonisomorphism $\tilde{f}:T_2\to T_1$ in $\mathcal{K}_S$ is given 
by an element $f$ in $e_1\rad S e_2$, and if $\tilde{f}$ factors through $\add (T)$, then it factors through 
$\add(T_2\oplus T_6)$, because every path in $Q_S$ from $1$ to $2$ which is passing through vertex $3$ it must went 
through vertex $2$ or $6$. Consequently, $Q_{21}\simeq e_1(\rad S)e_2/e_1(\rad^2S)e_2$, and hence, there is exactly 
$q_{12}=1$ arrow $1\to 2$ in $Q_F$, denoted by $\alpha_1$ (the same as corresponding homomorphism $\alpha_1=\Hom_{\mathcal{K}_S}(T,\tilde{\alpha_1}):\tilde{P}_2\to\tilde{P}_1$). Similar arguments imply the remaining 
equalities and the following ones $q_{13}=q_{21}=q_{26}=q_{32}=q_{36}=q_{61}=q_{62}=0$. Note also that directly from Cartan matrix of $F$ one may read that $q_{24}=q_{25}=q_{42}=q_{45}=q_{46}=q_{52}=q_{54}=q_{56}=q_{64}=q_{65}=0$.

Further, it follows from Lemma \ref{tech:3.2}b), that the uniuqe (up to scalar) nonzero morphism $T_3\to T_4$ 
(respectively, $T_3\to T_5$) in $\mathcal{K}_S$ is induced by identity on $P_3$, which will be denoted, respectively, 
by $\tilde{\beta_4}:T_3\to T_4$ and $\tilde{\beta_2}:T_3\to T_5$. It is easy to check that those maps do not 
factorize through $\add(T)$, so we have irreducible morphisms $\beta_4:\tilde{P}_3\to\tilde{P}_4$ and $\beta_2:
\tilde{P}_3\to\tilde{P}_5$ in $\proj F$, and hence arrows $\beta_4:4\to 3$ and $\beta_2:5\to 3$ in $Q_F$ (in particular, $q_{43}=q_{53}=1$). Next, observe that any nonzero morphism $T_1\to T_4$ (respectively, $T_1\to T_5$) 
is nonisomorphism and it factorizes through $\tilde{\beta_4}$ (respectively, through $\tilde{\beta_2}$), hence 
through $T_3$, and therefore $q_{41}=q_{51}=0$.

Next, consider the following relations in $S$ 
$$\beta_1=\rho\omega-\alpha\beta,\beta_3=\beta_1-\lambda(\alpha\beta\gamma\sigma)^{m-1}\alpha\beta\in e_1Se_3$$ 
and $\gamma_1=\nu\delta,\gamma_2=\gamma\sigma+\lambda(\gamma\sigma\alpha\beta)^{m-1}\gamma\sigma\in e_3Se_1$. 
Then the corresponding homomorphisms $\beta_1,\beta_3:P_3\to P_1$ and $\gamma_1,\gamma_2:P_1\to P_3$ in $\proj S$ induce morphisms $\tilde{\beta_1}:T_5\to T_1$, $\tilde{\beta_3}:T_4\to T_1$ and $\tilde{\gamma_1}, 
\tilde{\gamma_2}:T_1\to T_3$ in $\mathcal{K}_S$. We claim that induced homomorphisms $\beta_1:\tilde{P}_5\to 
\tilde{P}_1, \beta_3:\tilde{P}_4\to\tilde{P}_1$ and $\gamma_1,\gamma_2:\tilde{P}_1\to\tilde{P}_3$ are irreducible 
in $\proj F$ and $q_{14}=q_{15}=1$ and $q_{31}=2$. 

Note that $c^F_{45}=c^F_{54}=0$, so directly from Lemma \ref{tech:3.2}e) we can deduce that there is a nonzero 
morphism $T_4\to T_i$ (respectively, $T_5\to T_i$) if and only if $i=3$ or $1$, and moreover, it is induced (up to 
scalar) respectively, by $X_3^m$, if $i=3$ and by $\tilde{\beta_3}$ (respectively, $\tilde{\beta_1}$), for $i=1$. In 
particular, it is easy to see that $\beta_1$ and $\beta_3$ are irreducible in $\proj F$, so $q_{14}=q_{15}=1$. 
Moreover, one may check that $\gamma_1\beta_3=-\lambda X_3^m$, hence morphisms $T_4\to T_3$ ($T_5\to T_3$) factors 
through $\tilde{\beta_3}$ ($\tilde{\beta_1}$), and therefore $q_{34}=q_{35}=0$. 

Now, observe that each linear generator $X^k_{31}$ of $e_3Se_1$, $k\in\{1,\dots,m-1\}$, is a path passing through 
vertex $2$ (or $6$), so the induced morphism $T_1\to T_3$ in $\mathcal{K}_S$ factors through $T_2$, and hence 
$q_{13}\leqslant 2$. It is not hard to prove that both $\widetilde{\gamma_1}$ and $\widetilde{\gamma_2}$ (vieved as 
morphisms $T_1\to T_3$ in $\mathcal{K}_S$) do not factorize through $\add(T)$ (it is enough to see that $\nu\delta,
\gamma\sigma$ are paths of lenght 2 in $KQ_S$, while morphisms $T_1\to T_3$ in $\mathcal{K}_S$ which factorize 
through a summand $T_i$ of $T$ are induced by combinations of paths of length at least 3). Consequently, 
corresponding homomorphisms $\gamma_1$ and $\gamma_2$ are irreducible in $\proj F$. Clearly, $\gamma_1,\gamma_2$ are 
linearly independent as elements in $e_3S e_1$, hence also as morphisms $T_1\to T_3$ in $\mathcal{K}_S$ (note that 
here $\Hom_{\mathcal{K}_S}(T_1,T_3)=\Hom_{\mathcal{C}_S}(T_1,T_3)=\Hom_S(P_1,P_3)$). As a result, we obtain that 
indeed $q_{31}=2$, and irreducible homomorphisms $\gamma_1,\gamma_2$ in $\proj F$ induce two arrows $\gamma_1,
\gamma_2:3\to 1$ in $Q_F$. 

To end, let us also mention that there are no loops in $Q_F$. Indeed, since there is no loop in any vertex 
$i\in\{1,2,3,6\}$ and every cycle $X_i^k$ in $i$ passes through $3$, we deduce that $q_{ii}=0$, for $i\in
\{1,2,3,6\}$. Moreover, easy calculation involving bases  of $e_4Se_4$ and $e_3Se_3$ shows that 
$\Hom_{\mathcal{C}_S}(T_4,T_4)$ admits a $K$-basis formed by chain maps $f^k=(f^k_1,f^k_0)$, for 
$k\in\{0,1,\dots,m-1\}$, $f_1$ and $f_0$, where $f^k$ is given by $f^k_1=X_4^k:P_4\to P_4$ and $f^k_0=X_3^k:P_3\to 
P_3$, while $f_1=(X_4^m,0)$ and $f_0=(0,X_3^m)$. One can directly verify that the following homotopies hold: $f_0\simeq f_1$ and $f^k\simeq 0$, for any $k\in\{1,\dots,m-1\}$ (see Lemma \ref{tech:3.2}d)). It follows that 
$\Hom_{\mathcal{K}_S}(T_4,T_4)$ has a basis given by two homotopy classes, one of $f_0$ and one of $f^0$. In 
particular, $f_0$ yields unique generator of $\rad_F(\tilde{P}_3,\tilde{P}_3)$ (over $K$), which factors through 
$T_3$ (even, through $\beta_4$), and hence $q_{44}=0$. Similar arguments show that $q_{55}=0$. This 
completes the proof. \end{proof} 

We get analogous result for the remaining algebra.

\begin{lem} The Gabriel quivers of $E$ and $E(m,\lambda)$ coincide. \end{lem}

\begin{proof} As before, we abbreviate $T=T^{(1)}$, and for any $i\in\{1,\dots,6\}$, by $\tilde{P}_i$ we denote the 
projective module $\Hom_{\mathcal{K}_S}(T,T_i)$ in $\ind E$ corresponding to $i$-th summand $T_i$ of $T$ (with order 
of summands as presented in previous section). First, let us consider the following homomorphisms in 
$\proj S$
$$\alpha_1=\alpha:P_2\to P_1,$$ 
$$\alpha_2=\beta:P_3\to P_2,$$
$$\alpha_3=\rho:P_6\to P_1,$$ 
$$\alpha_4=\omega:P_3\to P_6,$$
$$\alpha_5=\gamma:P_4\to P_3 \quad\mbox{and}\quad\alpha_6=\sigma:P_1\to P_4 $$ 
induced by arrows in $\alpha,\beta,\rho,\omega,\gamma$ and $\sigma$ in $Q_S$. It is clear from the last proof that for any $i\in\{1,\dots,6\}$ 
homomorphism $\alpha_i:P_k\to P_l$ induce morphism $\tilde{\alpha_i}:T_k\to T_l$ which does not factor through $\add(T)$, and hence homomorphism 
$\alpha_i=\Hom_{\mathcal{K}_S}(T,\tilde{\alpha_i}):\tilde{P}_k\to\tilde{P}_l$ is irreducible in $\proj E$.

Moreover, let $\gamma_0:P_1\to P_3$ be a homomorphism in $\proj S$ given by path $\gamma_1=\nu\delta$ in the notation from the proof of previous Lemma. Using arguments presented there, one may easily see that $\gamma_0$ induces unique 
(up to scalar) morphism $\tilde{\gamma_0}:T_1\to T_3$ which does not factor through $\add(T)$ ($\gamma_2$ gives 
morphism which factors through $T_4=\textbf{P}_4$). In consequence, corresponding homomorphism $\gamma_0=
\Hom_{\mathcal{K}_S}(T,\tilde{\gamma_0}):\tilde{P}_1\to\tilde{P}_3$ is irreducible in $\proj E$, and hence, we have 
exactly one arrow $1\to 3$ in $Q_E$, denoted also by $\gamma_0$. As in the previous proof we can now deduce that 
$\alpha_1,\dots,\alpha_6$ and $\gamma_0$ are all arrows in $Q_E$ between vertices $i,j\neq 5$. 

Finally, we can repeat arguments used in the last proof and conclude that the remaining arrows $\beta_1:1\to 5$ 
and $\beta_2:5\to 3$ in $Q_E$ are induced from morphisms $\tilde{\beta_1}:T_5\to T_1$ and $\tilde{\beta_2}:T_3\to 
T_5$ given, respectively, by $\beta_1=\rho\omega-\alpha\beta$ and $\beta_2$ the identity on $P_3$. Now, the 
required claim follows.\end{proof}

Now, let us present two lemmas describing the relations in the both considered algebras. 

\begin{lem} The following relations hold in $F$
\begin{enumerate}
\item[(1)] $\alpha_1\alpha_2-\alpha_3\alpha_4+\beta_1\beta_2=0$ and $\alpha_1\alpha_2-\alpha_3\alpha_4+\beta_3\beta_4+
\lambda(\alpha_3\alpha_4\gamma_1)^{m-1}\alpha_3\alpha_4=0$.

\item[(2)] $\alpha_2\gamma_1=\alpha_2\gamma_2$ and $\gamma_2\alpha_1=\gamma_1\alpha_1$.

\item[(3)] $\alpha_4\gamma_2=\alpha_4\gamma_1+\lambda(\alpha_4\gamma_1\alpha_3)^{m-1}\alpha_4\gamma_1$ and $\gamma_2\alpha_3=\gamma_1\alpha_3+
\lambda(\gamma_1\alpha_3\alpha_4)^{m-1}\gamma_1\alpha_3$.
 
\item[(4)] $\beta_2\gamma_1=\beta_4\gamma_2=\gamma_1\beta_1=\gamma_2\beta_3=0$.

\item[(5)] $\beta_2\gamma_2\beta_1\beta_2=\beta_4\gamma_1\beta_3\beta_4=0$.

\item[(6)] $\gamma_1\alpha_1\alpha_2=\gamma_1\alpha_3\alpha_4$ and $\alpha_1\alpha_2\gamma_1=\alpha_3\alpha_4\gamma_1$.

\item[(7)] $(\alpha_3\alpha_4\gamma_1)^m\alpha_1=(\alpha_3\alpha_4\gamma_1)^m\alpha_3=\gamma_1(\alpha_3\alpha_4\gamma_1)^m=0$ and 
$\alpha_2(\gamma_1\alpha_3\alpha_4)^m=\alpha_4(\gamma_1\alpha_3\alpha_4)^m=0$.

\end{enumerate}\end{lem}

\begin{proof} The relations in (1) follow directly from definition of irreducible homomorphisms corresponding to 
arrows in $Q_F$. Easy calculation shows that $\alpha_2\gamma_1-\alpha_2\gamma_2=\beta\nu\delta-\beta\gamma\sigma-
\lambda(\beta\gamma\sigma\alpha)^{m-1}\beta\gamma\sigma=0$ and $\gamma_1\alpha_1-\gamma_2\alpha=\nu\delta\alpha-
\gamma\sigma\alpha-\lambda(\gamma\sigma\alpha\beta)^{m-1}\gamma\sigma\alpha=0$, by Section \ref{sec:2} (1)-(2). 
Therefore (2) holds. Further 
$$\alpha_4\gamma_2-\alpha_4\gamma_1-\lambda(\alpha_4\gamma_1\alpha_3)^{m-1}\alpha_4\gamma_1=
\omega\gamma\sigma-\omega\nu\delta+\lambda\omega(\gamma\sigma\alpha\beta)^{m-1}\gamma\sigma-$$
$$-\lambda(\omega\nu\delta\rho)^{m-1}\omega\nu\delta=0,\qquad\qquad\qquad\qquad\qquad\qquad$$ 
because in S, we have $\omega(\gamma\sigma\alpha\beta)^{m-1}\gamma\sigma=\omega(\nu\delta\rho\omega)^{m-1}\gamma\sigma=(\omega\nu\delta\rho)^{m-1} \omega\gamma\sigma$ and $\omega\gamma\sigma=\omega\nu\delta$, by Section 
\ref{sec:2} (7). In a similar way one may chceck that the second relation in (3) follows from (8) in Section 
\ref{sec:2}.

For (4) observe that since $\beta_2$ and $\beta_4$ are induced by indentity on $P_3$, we may identify relations $r_1=\beta_2\gamma_1$ and $r_2=\beta_4\gamma_2$ in $E$ with relations given by $\gamma_1=\nu\delta\in \nu S$ and 
$\gamma_2=\gamma(\sigma+\lambda(\sigma\alpha\beta\gamma)^{m-1}\sigma) \in\gamma S$, respectively. As a result 
corresponding morphisms $\tilde{r_1}:T_1\to T_5$ and $\tilde{r_2}:T_1\to T_4$ are homotopic to $0$, by Lemma 
\ref{tech:3.2}b). Note also that $\gamma_1\beta_1=\nu\delta(\rho\omega-\alpha\beta)=\nu(\delta\rho\omega-\delta\alpha\beta)=0$, due to Section \ref{sec:2} (6) and $\gamma_2\beta_1=(\gamma\sigma+\lambda X^{m-1}_{31})(\rho\omega-
\alpha\beta)=\gamma(\sigma\rho\omega-\sigma\alpha\beta)+\lambda X_3^{m-1}\gamma(\sigma\rho\omega-\sigma\alpha\beta)=
\lambda X_3^m+\lambda^2 X_3^{2m-1}=\lambda X_3^m$, by Section \ref{sec:2} (3), (10) and remarks. Now it is clear that 
$$\gamma_2\beta_3=\gamma_2(\beta_1-\lambda X_{13}^{m-1})=\lambda X_3^m - \lambda(\gamma\sigma+\lambda X^{m-1}_{31})X_{13}^{m-1}$$
is equal to $0$ in $F$, because $\gamma\sigma X_{13}^{m-1}=\gamma\sigma X_1^{m-1}\alpha\beta=X_3^m$, while $X^{m-1}_{31}X_{13}^{m-1}=X_3^{m-1}\gamma\sigma X_1^{m-1}\alpha\beta=X_3^{2m-1}=0$, again by Section \ref{sec:2} (10) (note: 
$m\geqslant 2$ implies $2m-1>m$). This proves that all relations in (4) hold. 

Moreover, relation $\beta_2\gamma_2\beta_1$ (respectively, $\beta_4\gamma_1\beta_3$) in $F$ corresponds to a 
morphism $T_5\to T_5$ (respectively, $T_4\to T_4$), which is given by single homomorphism $P_3\to P_3$ identified, 
respectively, with $\gamma_2\beta_1=(\gamma\sigma+\lambda(\gamma\sigma\alpha\beta)^{m-1}\gamma\sigma)(\rho\omega-
\alpha\beta)=\lambda X_3^m$ (by Section \ref{sec:2} (3)) and with $\gamma_1\beta_3=-\lambda\nu\delta
(\alpha\beta\gamma\sigma)^{m-1}\alpha\beta=-\lambda X_3^m$ (by (4) above; see also remarks in Section \ref{sec:2}). 
Therefore, using Lemma \ref{tech:3.2}b) we deduce that both relations $r_1=\beta_2\gamma_2\beta_1\beta_2$ and 
$r_2=\beta_4\gamma_1\beta_3\beta_4$ hold in $F$, since corresponding morphims $\tilde{r_1}:T_3\to T_5$ and 
$\tilde{r_2}:T_3\to T_4$ are induced by a cycle in $e_3Se_3$ which factorizes both through $\gamma$ and $\nu$, 
respectively. This implies (5). 

Finally, we will provide arguments for relations in (6) and (7). First, as an immediate consequence of (5)-(6) in 
Section \ref{sec:2}, we obtain that $\alpha_1\alpha_2\gamma_1-\alpha_3\alpha_4\gamma_1=(\alpha\beta-\rho\omega)
\nu\delta=0$ and $\gamma_1\alpha_1\alpha_2-\gamma_1\alpha_3\alpha_4=\nu\delta(\alpha\beta-\rho\omega)=0$, and 
hence (6) follows. Now, recall that the cycle $\alpha_3\alpha_4\gamma_1$ (respectively, $\gamma_1\alpha_3\alpha_4$) 
in $F$ corresponds to the cycle $\tilde{X}_1$ (respectively, $\tilde{X}_3$) in $S$ (in the notation from Section 2), so applying Section \ref{sec:2} (5$'$), we conclude that $(\alpha_3\alpha_4\gamma_1)^k=X_1^k$ and 
$(\gamma_1\alpha_3\alpha_4)^k=X_3^k$, for all $k\in\{2,\dots,m\}$. In particular, then relations 
$(\alpha_3\alpha_4\gamma_1)^m\alpha_1=(\alpha_3\alpha_4\gamma_1)^m\alpha_3=\gamma_1(\alpha_3\alpha_4\gamma_1)^m=0$ 
in $F$ follow from relations (9), (1$'$) and (4$'$) in Section \ref{sec:2}. The remaining relations from (7) are 
easy corollaries from Section \ref{sec:2} (2$'$) and the fact that $\beta X_3^m=0$ in $S$. The proof is now 
finished. \end{proof}

\begin{lem} The following relations hold in $E$
\begin{enumerate}
\item[(8)] $\alpha_1\alpha_2-\alpha_3\alpha_4+\beta_1\beta_2$,

\item[(9)] $\alpha_3\alpha_4\alpha_5=\alpha_1\alpha_2\alpha_5+\lambda(\alpha_1\alpha_2\alpha_5\alpha_6)^{m-1}\alpha_1\alpha_2\alpha_5$, 

\item[(10)] $\gamma_0\beta_1=0$, $\gamma_0\alpha_3=\alpha_5\alpha_6\alpha_3$ and $\gamma_0\alpha_1=\alpha_5\alpha_6\alpha_1+
\lambda(\alpha_5\alpha_6\alpha_1\alpha_2)^{m-1}\alpha_5\alpha_6\alpha_1$,

\item[(11)] $\alpha_2\gamma_0=\alpha_2\alpha_5\alpha_6+\lambda(\alpha_2\alpha_5\alpha_6\alpha_1)^{m-1}\alpha_2\alpha_5\alpha_6$ and 
$$(\alpha_2\alpha_5\alpha_6\alpha_1)^{m-1}\alpha_2\alpha_5\alpha_6\alpha_3=0,$$

\item[(12)] $\alpha_4\gamma_0=\alpha_4\alpha_5\alpha_6$ and $\alpha_4\alpha_5\alpha_6\alpha_1(\alpha_2\alpha_5\alpha_6\alpha_1)^{m-1}=0$,

\item[(13)] $\alpha_6\alpha_3\alpha_4=\alpha_6\alpha_1\alpha_2+\lambda(\alpha_6\alpha_1\alpha_2\alpha_5)^{m-1}\alpha_6\alpha_1\alpha_2$,

\item[(14)] $\beta_2\gamma_0=0$, $\beta_2\alpha_5\alpha_6\alpha_1=\beta_2\alpha_5\alpha_6\alpha_3=0$ and 
$\beta_2\alpha_5\alpha_6\beta_1\beta_2=0$,

\item[(15)] if a cycle $C$ at vertex $i$ in $Q_E$ is of length $3$ or $4$, then $\theta C^m=C^m\phi=0$, for 
all arrows $\theta,\phi$ with $t(\theta)=s(\phi)=i$.\end{enumerate}\end{lem}

\begin{proof} (8) is obvious from definition of irreducible homomorphisms corresponding to arrows of $Q_E$. For 
(9) it suffices to use Section \ref{sec:2} (4), whereas (10) is an immediate consequence of (6), (8) and (2) in 
Section \ref{sec:2}, respectively. Further (11) follows from Section \ref{sec:2} (1) and (1$'$). To see (12) it 
is enough to apply Section \ref{sec:2} (7) and (2'). Now, Section \ref{sec:2} (3) implies (13), while (14) is easily 
obtained from Lemma \ref{tech:3.2}b), and (8), (2) and (3) in Section \ref{sec:2}. Finally, to prove (15) we need 
only to observe that any cycle in $Q_E$ of length $3$ or $4$ is given by a single cycle in $Q_S$ of length $4$ or 
$4m$. Then this condition holds by (14) above and remarks in Section \ref{sec:2} (see also previous proof).
\end{proof}

Now, we may finish the proof by showing appropriate bases of algebras $E(m,\lambda)$ and $F(m,\lambda)$.

\begin{lem} $C_{F(m,\lambda)}=C_{F}$ and $\dim_K F(m,\lambda)=16m+16$
\end{lem}

\begin{proof} We abbreviate $F=F(m,\lambda)$ for most part. First of all, observe that relations (F1)-(F7) 
defining $F(m,\lambda)$ imply also the following relations  
\begin{enumerate}
\item[(1')] $\alpha_3\alpha_4\gamma_1\alpha_1\alpha_2=\alpha_3\alpha_4\gamma_1\alpha_3\alpha_4$ and $\alpha_3\alpha_4\gamma_1\beta_1=\alpha_3\alpha_4\gamma_1\beta_3=0$,

\item[(2$'$)] $\alpha_2\gamma_1\beta_1=\alpha_2\gamma_1\beta_3=0$ and $(\alpha_2\gamma_1\alpha_1)^{m-1}\alpha_2\gamma_1\alpha_3=0$,
  
\item[(3$'$)] $\gamma_1\alpha_3\alpha_4\gamma_1\beta_1=\gamma_1\alpha_3\alpha_4\gamma_1\beta_3=0$ and $\gamma_1\alpha_3\alpha_4\gamma_1=
\gamma_1\alpha_3\alpha_4\gamma_2$,

\item[(4$'$)] $\beta_4\gamma_1\alpha_1=\beta_4\gamma_1\alpha_3=\beta_4\gamma_1\beta_1=0$,

\item[(5$'$)] $\beta_2\gamma_2\alpha_1=\beta_2\gamma_2\alpha_3=\beta_2\gamma_2\beta_3=0$,

\item[(6$'$)] $\alpha_4\gamma_1\beta_1=\alpha_4\gamma_1\beta_3=0$ and $(\alpha_4\gamma_1\alpha_3)^{m-1}\alpha_4\gamma_1\alpha_1=0$,

\item[(7$'$)] for every vertex $i$ and any cycle $C$ of length in $Q_F$ with source and terget $i$, we have $C^m\phi=0=\theta C^m$, for all arrows $\phi,\theta$ in $Q_E$ with $t(\theta)=s(\phi)=i$.\end{enumerate}

We shall give some brief comments on the above relations. Note first that (4$'$) and (5$'$) follow from (2), (4) and second relation in (3). Moreover, first relation in (1$'$) is a direct consequence of (6), while the rest follow from 
(6), (2) and (4). From (2) and (4) we may easily deduce first two relations in (2$'$), and for the remaining one, 
observe that (2), (3) and (6) imply 
$$\lambda(\alpha_2\gamma_1\alpha_1)^{m-1}\alpha_2\gamma_1\alpha_3=\lambda \alpha_2(\gamma_1\alpha_1\alpha_2)^{m-1}\gamma_1\alpha_3=
\alpha_2(\lambda(\gamma_1\alpha_3\alpha_4)^{m-1}\gamma_1\alpha_3)=$$ 
$$ =\alpha_2(\gamma_2\alpha_3-\gamma_1\alpha_3)=(\alpha_2\gamma_2-\alpha_2\gamma_1)\alpha_3 =0,$$ 
so this also holds. In a similar way one may postmultiply the first relation from (3) by $\alpha_1$ and use (2) to 
obtain that the last relation from (6$'$) holds. Now, applying (4), we get the first relation from (6$'$) and the 
second is obtained by using (4) and the first relation in (3) twice, and then (7). Finally, first two relations in 
(3$'$) are consequences of (6$'$), and the last follows from (7), because, by (3), the difference 
$\alpha_3\alpha_4\gamma_1-\alpha_3\alpha_4\gamma_2$ is $\lambda(\alpha_3\alpha_4\gamma_1)^m$. It is enough to show 
that (7$'$) holds. Since (2) and (5) hold in $F$ we may restrict our considerations to cycles in vertex $i\in 
\{1,2,3,6\}$. On the other hand, if $C$ and $C'$ are two different cycles in $Q_F$ of length $3$ in vertex $i\in 
\{1,2,3,6\}$, then it follows from (3) and (2) that either $C=C'$ or $C-C'$ is $m$-th power of $C$ or $C'$ up to 
scalar multiplication, and consequently, we get $C^m=(C')^m$. Therefore, we only need to check condition (7$'$) for 
one choosen cycle $C$ in a given vertex $i$. If $C=\alpha_3\alpha_4\gamma_1$ is a cycle in $i=1$, then $\gamma_2C^m=
\gamma_2(\alpha_1\alpha_2\gamma_1)^m=\gamma_1C^m$, by (2), so combining this with (7) and (1$'$), we conclude that 
$C^m\phi=\theta C^m=0$, for all arrows $\phi$ (respectively, $\theta$) with source (respectively, target) in $i=1$. 
Because $\gamma_1\alpha_3\alpha_4=\gamma_2\alpha_1\alpha_2$, due to (6) and (2), one may deduce from (4), (7) and 
previous case ($i=1$), that (7$'$) is satisfied also for the cycle $C=\gamma_1\alpha_3\alpha_4$ in vertex $3$. The 
remaining cases follows from previous ones.
 
Let $\tilde{e}_1,\dots,\tilde{e}_6$ denote all primitive idempotents in $F$ corresponding to vertices of $Q_F$. Now we will explicitly describe base $\mathcal{B}_i$, for every projective module $\tilde{P}_i:=\tilde{e}_iF$ in 
$\ind F$, $i\in\{1,\dots,6\}$. 

First observe that all paths in $Q_F$ of length $1$ starting in vertex $1$ are precisely the arrows $\alpha_1,\alpha_3,\beta_1$ and $\beta_3$. The Paths of length $2$ are $\alpha_1\alpha_2, \alpha_3\alpha_4, \beta_1\beta_2$ and 
$\beta_3\beta_4$, but vieved as elements of $F$ both $\beta_1\beta_2$ and $\beta_3\beta_4$ are linear combinations of 
$\alpha_1\alpha_2, \alpha_3\alpha_4$ and $(\alpha_3\alpha_4\gamma_1)^{m-1}\alpha_3\alpha_4$, by (1). Next, it is easy 
to see that all paths in $Q_F$ of length $3$ starting in $1$ are equal modulo $I$ to a linear combination of cycles 
$\alpha_3\alpha_4\gamma_1$ and $(\alpha_3\alpha_4\gamma_1)^m$. On the other hand, by (7$'$) and (1$'$) above, every 
path in $Q_F$ of the form $(\alpha_3\alpha_4\gamma_1)^r\phi$ with $r\geqslant 1$ and $\phi=\beta_1$ or $\beta_3$, is 
eual to $0$ in $F$, as well as every path of the form $(\alpha_3\alpha_4\gamma_1)^m\psi$, for an arrow $\psi$ in 
$Q_F$ with source in $1$. Moreover, we have $(\alpha_3\alpha_4\gamma_1)^k\alpha_1\alpha_2= 
(\alpha_3\alpha_4\gamma_1)^k\alpha_3\alpha_4$, for $r\in\{1,\dots,m-1\}$, also by (1$'$), and consequently, the 
following subset of $F$ 
$$\mathcal{B}_1=\{\beta_1,\beta_3,\alpha_1\alpha_2,(\alpha_3\alpha_4\gamma_1)^m\}\cup\{(\alpha_3\alpha_4\gamma_1)^r\phi;\quad \phi\in\mathcal{B}_1',
0\leqslant r\leqslant m-1\},$$
where $\mathcal{B}_1'=\{\epsilon_1,\alpha_1,\alpha_3,\alpha_3\alpha_4\}$, is a $K$-basis of projective module 
$\tilde{P}_1=\tilde{e}_1F$ in $\ind F$. Using similar arguments involving (7$'$) and (3$'$) we conclude that 
projective module $\tilde{P}_3$ in $\ind F$ has $K$-basis given as follows  
$$\mathcal{B}_3=\{\gamma_2,\gamma_2\beta_1,\gamma_1\beta_3,(\gamma_1\alpha_3\alpha_4)^m\}\cup\{(\gamma_1\alpha_3\alpha_4)^r\phi;\quad \phi\in\mathcal{B}_3',
0\leqslant r\leqslant m-1\},$$ 
with $\mathcal{B}_3'=\{\epsilon_3,\gamma_1,\gamma_1\alpha_3,\gamma_1\alpha_1\}$. As a result, the first and the third 
column of $C_F$ is of the form $\left[\begin{array}{cccccc} m+1 & m & m+1 & 1 & 1 & m \end{array}\right]^t$.

Further, it is easy to verify that relations presented in (7$'$), (2$'$) and (6$'$) (see also (2)-(3)) imply that 
indecomposable projective modules $\tilde{P}_2$ and $\tilde{P}_6$ in $\ind F$ admit bases defined in the following 
way
$$\mathcal{B}_2=\{(\alpha_2\gamma_1\alpha_1)^r\phi;\quad \phi\in\mathcal{B}_2',0\leqslant r\leqslant m-1\} \cup\{(\alpha_2\gamma_1\alpha_1)^m\} \cup$$
$$\cup \{(\alpha_2\gamma_1\alpha_1)^r\alpha_2\gamma_1\alpha_3\;\quad 0\leqslant r \leqslant m-2\},$$
and 
$$\mathcal{B}_6=\{(\alpha_4\gamma_1\alpha_3)^r\phi;\quad \phi\in\mathcal{B}_6',0\leqslant r\leqslant m-1\} \cup\{(\alpha_4\gamma_1\alpha_3)^m\} \cup $$
$$\cup\{(\alpha_4\gamma_1\alpha_3)^r\alpha_4\gamma_1\alpha_1\;\quad 0\leqslant r \leqslant m-2\},$$
where $\mathcal{B}_2'=\{\epsilon_2,\alpha_2,\alpha_2\gamma_1\}$ and $\mathcal{B}_6'=\{\epsilon_6,\alpha_4,\alpha_4\gamma_1\}$. 

Finally, as a direct consequence of (4$'$), (5$'$) and (4)-(5) we obtain that bases of $\tilde{P}_4$ and $\tilde{P}_5$ are, respectively, of the form 
$$\mathcal{B}_4=\{\beta_4\gamma_1, \beta_4, \epsilon_4, \beta_4\gamma_1\beta_3\}$$ 
and 
$$\mathcal{B}_5=\{\beta_2\gamma_2, \beta_2, \epsilon_5, \beta_2\gamma_2\beta_1\}.$$ 
The proof is now complete. \end{proof}

Finally, let us describe Cartan matrix of algebra $E(m,\lambda)$.

\begin{lem} $C_{E(m,\lambda)}=C_{E}$ and $\dim_K E(m,\lambda)=25m+13$. \end{lem}

\begin{proof} As in previous proof we denote the algebra $E(m,\lambda)$ by $E$, by $\tilde{e}_i=\epsilon_i+ I_E$ the idempotent in $E$ corresponding to trivial path $\epsilon_i$ in $Q_E$, and by $\tilde{P}_i$ the projective module 
$\tilde{e}_iE$ in $\ind E$. First, postmultiplying (8) and using (14) we obtain 
$\alpha_1\alpha_2\gamma_0-\alpha_3\alpha_4\gamma_0=-\beta_1\beta_2\gamma_0=0$ and so 
$\alpha_1\alpha_2\gamma_0=\alpha_3\alpha_4\gamma_0=\alpha_3\alpha_4\alpha_5\alpha_6$ in $E$, by (12). Further, 
by (9), we have 
$$\alpha_3\alpha_4\alpha_5\alpha_6=\alpha_1\alpha_2\alpha_5\alpha_6+\lambda(\alpha_1\alpha_2\alpha_5\alpha_6)^m,$$ 
and then (15) implies that $\alpha_1\alpha_2\alpha_5\alpha_6\beta_1=\alpha_3\alpha_4\alpha_5\alpha_6\beta_1= 
\alpha_3\alpha_4\gamma_0\beta_1=0$, due to (10). Moreover, (13) gives $\alpha_5\alpha_6\alpha_3\alpha_4-
\alpha_5\alpha_6\alpha_1\alpha_2=\lambda(\alpha_5\alpha_6\alpha_1\alpha_2)^m$, hence using (15) again we deduce 
that $\alpha_1\alpha_2\alpha_5\alpha_6\alpha_1\alpha_2=\alpha_1\alpha_2\alpha_5\alpha_6\alpha_3\alpha_4$. It follows 
that $K$-basis of $\tilde{P_1}$ is of the form
$$\mathcal{B}_1=\{(\alpha_1\alpha_2\alpha_5\alpha_6)^r\psi;\quad \psi\in\mathcal{B}_1', 0\leqslant r\leqslant m-1\}\cup \{\alpha_3\alpha_4, \beta_1, (\alpha_1\alpha_2\alpha_5\alpha_6)^m\},$$ 
where $\mathcal{B}_1'=\{\epsilon_1,\alpha_1,\alpha_1\alpha_2, \alpha_1\alpha_2\alpha_5, \alpha_3\}$. Now, applying 
(8), (10) and (13)-(15) we may similarily prove that 
$$\gamma_0\alpha_1\alpha_2=\gamma_0\alpha_3\alpha_4=\alpha_5\alpha_6\alpha_3\alpha_4=\alpha_5\alpha_6\alpha_1\alpha_2+\lambda(\alpha_5\alpha_6\alpha_1\alpha_2)^m,$$ 
$\alpha_5\alpha_6\alpha_1\alpha_2\gamma_0=\alpha_5\alpha_6\alpha_1\alpha_2\alpha_5\alpha_6$ and 
$\alpha_5\alpha_6\alpha_3\alpha_4\alpha_5\alpha_6\beta_1=0$. As a result, we obtain that basis of $\tilde{P_3}$ 
has analogous form
$$\mathcal{B}_3=\{(\alpha_5\alpha_6\alpha_1\alpha_2)^r\psi;\quad \psi\in\mathcal{B}_3',0\leqslant r\leqslant m-1\} \cup \{\gamma_0,\alpha_5\alpha_6\beta_1,(\alpha_5\alpha_6\alpha_1\alpha_2)^m\},$$ 
with $\mathcal{B}_3'=\{\epsilon_3,\alpha_5,\alpha_5\alpha_6, \alpha_5\alpha_6\alpha_1, \alpha_5\alpha_6\alpha_3\}$. 
Next, we may see from (11) that $\alpha_2\gamma_0\alpha_1=\alpha_2\alpha_5\alpha_6\alpha_1+\lambda(\alpha_2\alpha_5\alpha_6\alpha_1)^m$ and 
$$\alpha_2\gamma_0\alpha_3-\alpha_2\alpha_5\alpha_6\alpha_3=
\lambda(\alpha_2\alpha_5\alpha_6\alpha_1)^{m-1}\alpha_2\alpha_5\alpha_6\alpha_3=0,$$ 
and consequently, we get that $\tilde{P_2}$ has basis of the form
$$\mathcal{B}_2=\{(\alpha_2\alpha_5\alpha_6\alpha_1)^r\psi;\quad \psi\in\mathcal{B}_2', 0\leqslant r\leqslant m-1\} 
\cup\{(\alpha_2\alpha_5\alpha_6\alpha_1)^m\} \cup $$
$$\cup\{(\alpha_2\alpha_5\alpha_6\alpha_1)^r\alpha_2\alpha_5\alpha_6\alpha_3;\quad 0\leqslant r\leqslant m-2)\},$$ 
where $\mathcal{B}_2'=\{\epsilon_2,\alpha_2,\alpha_2\alpha_5,\alpha_2\alpha_5\alpha_6\}$. This is also clear from 
(12) that $\tilde{P_6}$ admits a basis given as follows
$$\mathcal{B}_6=\{(\alpha_4\alpha_5\alpha_6\alpha_3)^r\psi;\quad \psi\in\mathcal{B}_6', 0\leqslant r\leqslant m-1\}
\cup\{(\alpha_4\alpha_5\alpha_6\alpha_3)^m\}\cup $$ 
$$\cup\{(\alpha_4\alpha_5\alpha_6\alpha_3)^r\alpha_4\alpha_5\alpha_6\alpha_1;\quad 0 \leqslant r\leqslant m-2\},$$ 
with $\mathcal{B}_6'=\{\epsilon_6,\alpha_4,\alpha_4\alpha_5,\alpha_4\alpha_5\alpha_6\}$. Simple computations 
involving (9) (or (13)) show that base of $\tilde{P_4}$ may be choosen in the following way
$$\mathcal{B}_4=\{(\alpha_6\alpha_1\alpha_2\alpha_5)^r\psi;\quad \psi\in\mathcal{B}_4', 0\leqslant r\leqslant m-1\}\cup \{\alpha_6\beta_1,(\alpha_6\alpha_1\alpha_2\alpha_5)^m\},$$ 
for $\mathcal{B}_4'=\{\epsilon_4,\alpha_6,\alpha_6\alpha_1,\alpha_6\alpha_1\alpha_2,\alpha_6\alpha_3\}$. Finally, 
using all relations from (14) we immediately obtain that $\tilde{P_5}$ has basis of the form
$$\mathcal{B}_5=\{\epsilon_5,\beta_2,\beta_2\alpha_5,\beta_2\alpha_5\alpha_6,\beta_2\alpha_5\alpha_6\beta_1\},$$
and the proof is now finished.\end{proof}

\end{document}